\documentclass[arxiv,reqno,twoside,a4paper,12pt]{amsart}
\usepackage{mltools}
\usepackage{mlmath62,mlthm10-REAR}

\usepackage{upref,amsmath,amssymb,amsthm,mathrsfs}
\usepackage{tikz, graphics, xypic, latexsym, tensor}
\usepackage{graphicx}
\usepackage[colorlinks=true,linkcolor=blue,citecolor=blue]{hyperref}

\usepackage[scaled]{helvet} % ss
\usepackage{courier} % tt
\usepackage[mathbf]{euler}
\usepackage{mllocal}

\numberwithin{equation}{section}

\begin{document}

\date{\today}

\title[Cheeger-M\"uller theorem for a wedges along a submanifold]
{Cheeger-M\"uller theorem for a wedge singularity along an embedded submanifold}

\author{Luiz Hartmann}
\address{Universidade Federal de S\~ao Carlos (UFSCar),
	Brazil}
\email{hartmann@dm.ufscar.br, luizhartmann@ufscar.br}
\urladdr{http://www.dm.ufscar.br/profs/hartmann}

\author{Boris Vertman}
\address{Universit\"at Oldenburg, Germany}
\email{boris.vertman@uni-oldenburg.de}
\urladdr{https://uol.de/boris-vertman/}

\thanks{Partially support by FAPESP: 2021/09534-4 and CAPES/DAAD:
8881.700909/2022-01}

\subjclass[2020]{Primary 58J52; Secondary 57Q10}
\keywords{wedge singularity, analytic torsion, Reidemeister
torsion}

\begin{abstract}
In this paper we equate the analytic and the intersection Reidemeister torsions
on spaces with a specific type of wedge singularities, which arise by turning the
disc cross-sections in the tubular neighborhood of an embedded submanifold
of even co-dimension into cones. Our result is related to a similar equation, in a setting disjoint from ours,
which was previously obtained by Albin, Rochon and Sher \cite{ARS}.

\end{abstract}

\maketitle
\tableofcontents

%%%%%%%%%%%%%%%%%%%%%%
\section{Introduction and statement of the main result}
%%%%%%%%%%%%%%%%%%%%%%%

%%%%%%%%%%%%%%%%%%%%%%
\subsection{Spaces with wedge singularities}
%%%%%%%%%%%%%%%%%%%%%%%
Such spaces are commonly known as spaces with non-isolated conical sin\-gu\-la\-ri\-ti\-es.
They are the fundamental examples of smoothly Thom-Mather stratified spaces with iterated cone-edge metrics.
We refer the reader to Albin, Leichtnam, Mazzeo and Piazza \cite[Definition 1 and 5]{ALMP1} for a careful
definition of those. Spaces with wedge singularities are stratified spaces of depth $1$ and we
recall their definition here explicitly.

\begin{definition}\label{wedge-def}
Let $\widetilde{M}$ be an oriented compact manifold with boundary $\partial M$, which is the total
space of a fibration $\phi: \partial M \to B$ with base $B$ and typical fibre $F$ compact
smooth manifolds. Write $M$ for the open interior of $\widetilde{M}$.
Consider a collar $\widetilde{\cU} \cong [0,1) \times \partial M$
of the boundary, write $x\in [0,1)$ for its radial coordinate and set
$$
\cU := \widetilde{\cU} \cap M \cong (0,1) \times \partial M.
$$
We say that $(M,g)$ is a space with a wedge singularity, if the Riemannian
metric $g$ in the open interior $M$ has the form $g = g_0 + h$, with the rigid part
$$
g_0\restriction \cU = dx^2 + x^2 \kappa + \phi^* g_B,
$$
where $g_B$ is a Riemannian metric on $B$, $\kappa$ is a symmetric positive definite
bilinear form on $T\partial M \cong \ker d\phi \oplus \phi^*TB$, restricting to a Riemannian
metric on fibres $\ker d\phi \cong TF$, and vanishing on the second component.
The higher order term $h$ satisfies $|h|_{g_0} = O(x)$
as $x \to 0$. See illustration of $(\cU, g)$ in Figure \ref{figure-wedge}.
\end{definition}

\begin{figure}[ht]
  \includegraphics[width=0.7\linewidth]{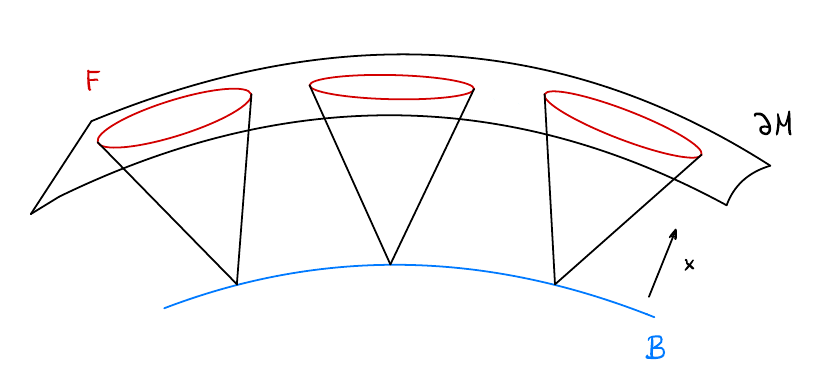}
  \caption{Illustration of the singular neighborhood $(\cU,g)$.}
\label{figure-wedge}
\end{figure}

\noindent \textbf{Notation:} We will use $\widetilde{M}, \overline{M}, M$ and $\widetilde{\cU}, \overline{\cU}, \cU$ as follows:
\begin{itemize}
\item $\widetilde{M}$ is compact with boundary $\partial M$ and collar $\widetilde{\cU} \cong [0,1) \times \partial M$
is a fibration of cylinders $[0,1) \times F$ over $B$.
\item $M$ is the open interior of $\widetilde{M}$ and $\cU = \widetilde{\cU} \cap M \cong (0,1) \times \partial M$
is a fibration of open cylinders $(0,1) \times F$ over $B$.
\item $\overline{M}$ is the metric space completion of $(M,d_g)$, where the distance metric $d_g$ is defined with respect to $g$.
Then $\overline{\cU}$ is a fibration of cones $C(F) = [0,1) \times F/ (\{0\} \times F)$ over $B$.
\end{itemize}

In the setup of  smoothly Thom-Mather stratified spaces
as in \cite{ALMP1}, $\widetilde{M}$ is the resolution of a compact stratified space $\overline{M}$.
Within the scope of this note, we will impose the following assumption, which
has been imposed before by Mazzeo and the second named author \cite{MaVe1} in their analysis of analytic torsion on
spaces with a wedge singularities.

\begin{assumption}\label{ass-high}
Consider local edge coordinates $(x,y,z)$ on $\cU$, where $y = (y_i)$
is the lift of a local coordinate system on $B$ and $z=(z_j)$ restrict
to local coordinates on each fibre $F$. We assume that the higher order term $h$ is "even", \ie
has an expansion in even powers of $x$ as $x\to 0$, when acting on the local vector fields $\{\partial_x, \partial_y, \partial_z\}$,
except for the cross-terms $h(\partial_x, \partial_y)$ and $h(\partial_x, \partial_z)$, which are assumed to have
an expansion in odd powers of $x$, as $x\to 0$.
\end{assumption}

This assumption is well-defined only within a particular \emph{even} equivalence class of coordinate charts in $\cU$.
A local coordinate system $(x', y', z')$ in $\cU$ is said to be in the even equivalence class of a coordinate chart $(x,y,z)$
if the asymptotics of $x'/x$, $y'$ and $z'$ near the wedge admits only powers of $x^2$, with coefficients in the
expansions depending smoothly on $y$ and $z$. We name coordinates within a fixed even equivalence class \emph{special}
coordinates and will stay within that class henceforth.

%%%%%%%%%%%%%%%%%%%%%%
\subsection{Spaces with wedges along a submanifold}
%%%%%%%%%%%%%%%%%%%%%%%
In this note we are actually interested in a very specific type of wedge singularities,
when $\overline{M}$ can be given a smooth structure and the singularity is entirely in the metric.
Such spaces have been studied for instance by Atiyah and Lebrun \cite{AL}, and their advantage lies in the fact
that can be deformed into smooth compact Riemannian manifolds. We define such spaces here,
taking reference to the classical tubular neighborhood theorem as in Lee \cite[Theorem 5.24]{Lee}.

\begin{figure}[ht]
	\includegraphics[width=0.7\linewidth]{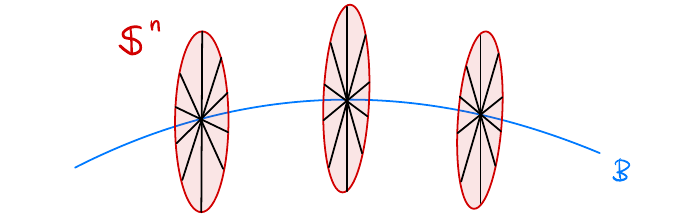}
	\caption{Illustration of the classical tubular neighborhood $(\overline{\cU},g_0)$.}
	\label{figure-embedded1}
\end{figure}

\begin{definition}\label{wedge-def2}
Consider an oriented compact Riemannian manifold $(\overline{M}^{\, m},g_0)$ with a
compact embedded submanifold $B^b \subset \overline{M}$. Consider an
open tubular neighborhood $\overline{\cU}  \subset \overline{M}$ of $B$,
the total space of a fibration $\pi: \overline{\cU} \to B$ with typical fibre $B_\delta(0) \subset \R^{m-b}$
for some $\delta > 0$. Denote by $x\in [0,1)$ the radial function of the disc. We write
$\phi:= \pi|_{\partial \overline{\cU}}$. We can always find $g_0$ such that
$$
g_0 \restriction \overline{\cU} = dx^2 + x^2 \kappa + \phi^*{g_0}|_B,$$
where $\kappa$ is a symmetric positive definite
bilinear form on $T\partial\overline{\cU} \cong \ker d\phi \oplus \phi^*TB$, restricting to
a sphere metric $g_{\mathbb{S}^n}$ on the fibres $\ker d\phi$, where $n=m-b-1$. We define a family of
wedge metrics on $M = \overline{M} \backslash B$ by specifying
$$
g_\varepsilon \restriction \overline{\cU} \backslash B = dx^2 + (1-\varepsilon^2) x^2 \kappa + \phi^*{g_0}|_B,
$$
and extending smoothly to $\overline{M} \backslash \overline{\cU}$. We call such $(M,g_\varepsilon)$
a smooth family of wedge spaces with singularity along an embedded submanifold.
\end{definition}

\begin{figure}[ht]
	\includegraphics[width=0.7\linewidth]{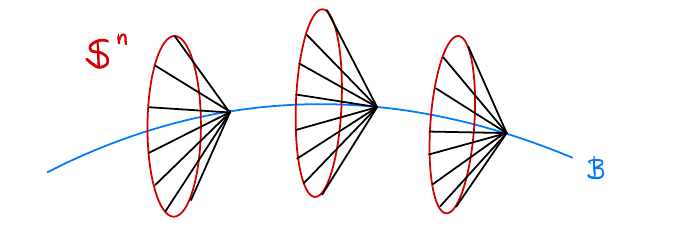}
	\caption{Illustration of the  tubular neighborhood $(\overline{\cU},g_\varepsilon)$.}
	\label{figure-embedded2}
\end{figure}

Note that each $(M,g_\varepsilon)$ is a space with a wedge singularity in the
sense of Definition \ref{wedge-def}, with $\phi$ being a fibration
of spheres $\mathbb{S}^n$.

%%%%%%%%%%%%%%%%%%%%%%
\subsection{Statement of the main result}
%%%%%%%%%%%%%%%%%%%%%%%

The main result of this note equates the analytic  and the intersection Reidemeister torsions
on smooth family $(M,g_\varepsilon)$ of wedge spaces with singularity along an embedded submanifold.
We twist by flat vector bundles associated to an orthogonal representation of the fundamental group $\pi_1(\overline{M})$,
but omit the vector bundle from the notation. A similar equality for different flat vector bundles, a situation
disjoint from ours, as explained in Remark \ref{ARS}, is obtained in \cite{ARS} using microlocal analysis and singular degeration arguments for general wedges of
Definition \ref{wedge-def} with $\dim F$ odd. Our result is as follows.

\begin{theorem}\label{main}
Consider an oriented compact Riemannian manifold $(\overline{M},g_0)$ with a
compact embedded submanifold $B \subset \overline{M}$. Consider the
family $(M,g_\varepsilon), \varepsilon \in (0,1)$ of wedge spaces with singularity along $B$, obtained by
deforming disc cross-sections in the tubular neighborhood $\overline{\cU}$ of $B \subset \overline{M}$
into cones. Assume that $\dim M$ and $\dim B$ are both odd.

\begin{enumerate}
\item Consider the analytic torsion $T(M,g_\varepsilon)$, defined in Section \ref{Sec-2}.
\item Consider the intersection Reidemeister torsion $I^{\mathfrak{p}}\tau(\overline{M}, h_M(\varepsilon))$ of the stratified space $\overline{M}$,
where $\mathfrak{p}$ is the upper or lower middle perversity, as defined in Section \ref{R-section} .
\end{enumerate}
Then
$$
T(M,g_\varepsilon) = I^{\mathfrak{p}}\tau(\overline{M}, h_M(\varepsilon)),
$$
where $h_M(\varepsilon) \in \det \, H_*(\overline{M})$ is chosen such that $R^{-1}(h_M(\varepsilon)) \in \det \, H^*_{(2)}(M,g_\varepsilon)$
is of $L^2$-norm one. The determinant line $\det \, H_*(\overline{M})$ and the integration map $R$ will be defined below in \S \ref{R-section} and \eqref{R}.
\end{theorem}

\begin{remark}\label{ARS}
Theorem \ref{main} holds with no changes, if we twist by a flat vector bundle $E$ associated to
an orthogonal representation
$$
\rho: \pi_1(\overline{M}) \to O(r,\R).
$$
Note that
the fibres $C(F_y), F_y \cong \mathbb{S}^n$ of the fibration $\phi: \overline{\cU} \to B$ retract to the cone tip $y \in B$, and hence
the restriction of $E$ to each fibre is trivial. In particular, in our setting, we always have
$$
H^*(F_y, E) = H^*(F_y, \R^r) \neq \{0\}.
$$
On the other hand, \cite{ARS} deals with representations of $\pi_1(M) \cong \pi_1(\widetilde{M})$ and assumes in
\cite[eq.\,(8)]{ARS} that $H^*(F_y, E) =  \{0\}$. Hence the set of cases, where our theorem applies, is disjoint from the cases,
where \cite[Theorem 1]{ARS} holds.
\end{remark}

Our proof works by relating $T(M,g_\varepsilon)$ to the classical
analytic torsion of $(\overline{M},g_0)$, introduced by Ray and Singer in
\cite{RS}, and $I^{\mathfrak{p}}\tau(\overline{M},h_M(\varepsilon))$ to the classical
Reidemeister torsion $\tau(\overline{M},h_M(\varepsilon))$ of $\overline{M}$, introduced by Reidemeister
\cite{Rei} and Franz \cite{Fra}. On the combinatorial side, we provide an argument independent of the
stratification invariance result by Dar \cite[Theorem 2.9]{Dar1}. We can then apply the result of Cheeger \cite{Che1} and M\"uller \cite{Mul}.
We have the following chain of equalities (we actually prove them for torsions as norms
on determinant line bundles in cohomology)

$$
T(M,g_\varepsilon) \overset{!}{=} T(\overline{M},g_0) = \tau(\overline{M}, h_M(0)) \overset{!}{=} I^{\mathfrak{p}}\tau(\overline{M},h_M(\varepsilon)).
$$

While the middle equality is the Cheeger-M\"uller theorem, the left and right equalities form the main contribution of this note.

\subsection*{Acknowledgments} We thank the referee for very useful remarks.

%%%%%%%%%%%%%%%%%%%%%%
\section{Analytic Ray-Singer torsion for wedges along a submanifold}\label{Sec-2}
%%%%%%%%%%%%%%%%%%%%%%%

Consider a wedge manifold $(M,g)$ with a wedge singularity at $B$ as in Definition \ref{wedge-def}.
Consider the differential forms $\Omega^*(M)$ and denote their completion with respect to the
$L^2$-scalar product by $L^2_*(M, g)$. Let $d^t_p$ denote the
formal adjoint of $d_p$, acting on $\Omega^p_0(M)$, and consider the
Hodge-Laplace operator
\begin{align*}
\Delta_p := d^t_p d_p + d_{p-1}
d^t_{p-1}: \Omega^p_0(M,E) \to \Omega^p_0(M,E).
\end{align*}
We define the maximal closed extension of $d_p$ in $L^2_p(M,g)$ with domain
\begin{equation}
\dom(d_{p, \max}) := \{\w \in L^2_p(M,g) \mid d \w \in L^2_{p-1}(M,g)\},
\end{equation}
where $d \w$ is defined distributionally.
We may also define the minimal closed extension of $d$ in $L^2_p(M,g)$ as the domain of
the graph closure of $d_p$ acting on smooth compactly supported functions
$\Omega^p_0(M)$. More precisely, the minimal domain is defined by
\begin{equation*}
\begin{split}
\dom(d_{p,\min}) := \{\w \in \dom(d_{p,\max}) \mid \exists (\w_n)_{n\in \N} \subset \Omega^p_0(M):
u_n:= \w - \w_n\\
\|u_n\|_{d_p} := \|d_p u_n\|^2_{L^2} + \|u_n\|^2_{L^2} \to 0 \ \textup{as} \ n\to \infty \}.
\end{split}
\end{equation*}
In the setting of isolated cones ($\dim B = 0$), the following was observed by
Cheeger \cite{Che2}, \cf also Br\"uning and Lesch \cite[Theorem 3.7, 3.8]{BL-1}:
for $\dim M-\dim B$ even, $\dom(d_{p,\min}) = \dom(d_{p,\max})$
for all $p$. Same holds for $\dim M-\dim B = 2\nu + 1$ odd, except for
$p = \nu$, in which case
$$
\dom(d_{\nu,\max}) / \dom(d_{\nu,\min}) \cong H^{\nu}(F).
$$
The same relations hold also for wedges by an argument, similar to the proof of
\cite[Lemma 2.2]{MaVe1}, where we identify $\dom(d_{p,\min})$ as conditions on
asymptotic expansions of elements in the maximal domain, that hold for isolated conical
singularities but also for wedges in the weak sense, \ie when paired with smooth functions on the wedge $B$.
We can now introduce relative and absolute self-adjoint extensions of $\Delta_p$
\begin{equation} \label{domains}\begin{split}
&\Delta_{p, \textup{rel}} := d^*_{p,\min} d_{p,\min} + d_{p-1, \min} d^*_{p-1, \min}, \\
&\Delta_{p, \textup{abs}} := d^*_{p,\max} d_{p,\max} + d_{p-1, \max} d^*_{p-1, \max}.
\end{split}
\end{equation}
One can also define, following \cite[(3.3a)]{BL-1}, the Friedrichs self adjoint extension of $\Delta$ by setting
$$
\Delta^{\mathscr{F}} := (d_p + d^t_{p-1})_{\max}  (d_p + d^t_{p-1})_{\min}
$$
As a consequence of \cite[Proposition 4.3]{MaVe1}
\begin{align}\label{zeta-def}
\zeta (s, \Delta^{\mathscr{F}}_p) :=
\frac{1}{\Gamma (s)} \int_0^\infty t^{s-1}
\left(\textup{Tr} e^{-t \Delta^{\mathscr{F}}_p} - \dim \ker \Delta^{\mathscr{F}}_p \right) dt, \,
\Re(s) \gg 0,
\end{align}
is well-defined and extends meromorphically
to $\C$, with $s=0$ being a regular point. We can therefore
proceed with the following

\begin{definition}\label{tors} The analytic torsion of
$(M,g)$ is denoted by $T(M,g)\in \R^+$
and defined by specifying its logarithmic value as
follows
\begin{align*}
\log T(M,g) &:= \frac{1}{2} \sum_{p=0}^m (-1)^{p} \, p \,
\left.\frac{d}{ds}\right|_{s=0} \zeta(s,\Delta^{\mathscr{F}}_{p})
\end{align*}
\end{definition}

Let $H^p_{(2)}(M,g) := \ker \Delta^{\mathscr{F}}_{p}$ denote the space
harmonic forms in $L^2_p(M,g)$. The
determinant line of $L^2$-cohomology is defined by
\begin{equation}\label{det-line}
\begin{split}
&\det \, H^*_{(2)}(M,g) := \bigotimes_{p=0}^m \det \, H^p_{(2)}(M,g)^{(-1)^{p+1}}, \\
&\textup{where} \ \det \, H^p_{(2)}(M,g):= \Lambda ^{\textup{top}} \Bigl( H^p_{(2)}(M,g) \Bigr),
\end{split}
\end{equation}
and $V^{-1}$ denotes the dual of a finite-dimensional vector space
$V$. The $L^2$-inner product of $L^2_*(M,g)$ yields a norm on
$H^*_{(2)}(M,g)$ and hence also on the determinant line $\det \, H^*_{(2)}(M,g)$, which we denote by
$\| \cdot \|_{\det \, H^*_{(2)}(M,g)}$.

\begin{definition}\label{tors-norm}
The Ray-Singer metric of $(M,g)$ is a norm on $\det \, H^*_{(2)}(M,g)$ given by
\begin{align*}
\|\cdot \|_{(M,g)}^{RS} := T(M,g) \| \cdot \|_{\det \, H^*_{(2)}(M,g)}.
\end{align*}
\end{definition}

\begin{theorem}\label{main-RS}
Consider a smooth family $(M,g_\varepsilon)$ of wedge spaces with singularity along $B$,
as in Definition \ref{wedge-def2}. Assume that $\dim M$ and $\dim B$ are both odd. Then
\begin{align}
\|\cdot \|_{(M,g_\varepsilon)}^{RS} = \|\cdot \|_{(M,g_0)}^{RS}.
\end{align}
\end{theorem}

\begin{proof}
Note that under the dimensional restriction of $\dim M-\dim B$ being even,
$\dom(d_{p,\max}) = \dom(d_{p,\min})$ for all degrees $p$. In particular,
the obtain equality for the relative, absolute and the Friedrichs self-adjoint extensions
of the Hodge Laplacian $\Delta_p$ in all degrees $p$
\begin{align}
\Delta_{p, \textup{rel}} = \Delta_{p, \textup{abs}}  = \Delta^{\mathscr{F}}_{p}.
\end{align}
We can now apply \cite[Theorem 1.4]{MaVe1} to $(M,g_\varepsilon)$. The metrics $g_\varepsilon$ satisfy Assumption \ref{ass-high}.
This implies that for $\dim M-\dim B$ even, the  metrics $\|\cdot \|_{(M,g_\varepsilon)}^{RS}$
do not depend on $\varepsilon \in [0,1)$, \ie
$$
\frac{d}{d\varepsilon} \|\cdot \|_{(M,g_\varepsilon)}^{RS} = 0.
$$
This proves the theorem.
\end{proof}

%%%%%%%%%%%%%%%%%%%%%%
\section{Intersection R-torsion for wedges along a submanifold}\label{R-section}
%%%%%%%%%%%%%%%%%%%%%%%

Consider $\overline{M}$ as a stratified  space with singular stratum $B$ and a tubular neighborhood $\overline{\cU}$
being a fibre bundle of spheres over $B$. In this section we study the intersection Reidemeister torsion of $\overline{M}$.
We begin with algebraic preliminaries, for a complete description see e.g. M\"uller \cite{Mul1}. \medskip

Let $V$ be a finite dimensional vector space. If ${\bf v}=\{v_1,\ldots,v_{\dim V}\}$ and
${\bf w}=\{w_1,\ldots,w_{\dim V}\}$ are two bases of $V$, we
denote by $({\bf w}/{\bf v})$ the matrix of base change from ${\bf v}$ to
${\bf w}$, so that for all $i=1,\ldots,\dim V$
$$
w_i = \sum\limits_{j=1}^{\dim V} ({\bf w}/{\bf v})_{ij} v_j.
$$
We denote by $[{\bf w}/{\bf v}]$ the absolute value of the
determinant of $({\bf w}/{\bf v})$. Consider a finite chain complex $\CC = (C_\bullet,\pl_\bullet)$ of finite dimensional
vector spaces
\begin{equation*}
	\xymatrix{ \mathfrak{C}: 0\ar[r] & C_m\ar[r]^{\pl_m}  & C_{m-1}
		\ar[r]^{\pl_{m-1}}
		& \cdots \ar[r]^{\pl_2} &
		C_{1} \ar[r]^{\pl_1} & C_0 \ar[r] &0}.
\end{equation*}
Define, for any $q=0,\ldots,m$,
$$
Z_q:= \ker \pl_q, \quad B_q := \Im \pl_{q+1}, \quad H_q(\CC) :=
Z_q/ B_q.
$$
Choose a preferred basis ${\bf c}_q$ for each $C_q$ and
a basis ${\bf h}_q$ for each $H_q(\CC)$. Choose representatives
$ {\bf \tilde h}_q$ in $C_q$ of the homology classes ${\bf h}_q$.
Choose a linearly independent set ${\bf b}_q \subset C_q$,
such that $\pl_q({\bf b}_{q})$ is a basis of $B_{q-1}$. Then
$(\pl_{q+1}({\bf b}_{q+1}) \ {\bf \tilde h}_q \ {\bf b}_{q})$ is a basis of $C_q$ and it makes
sense to consider $[\pl_{q+1}({\bf b}_{q+1})\; {\bf \tilde h}_q\;{\bf b}_{q}/{\bf c}_q]$.
We will denote this positive number by
$$
[\pl_{q+1}({\bf b}_{q+1})\; {\bf h}_q\;{\bf b}_{q}/{\bf c}_q]:=
[\pl_{q+1}({\bf b}_{q+1})\; {\bf \tilde h}_q\;{\bf b}_{q}/{\bf c}_q],
$$
as it depends only on ${\bf b}_q$, ${\bf b}_{q+1}$ and ${\bf h}_q$,
and is independent of the choice ${\bf \tilde h}_q$.

\begin{definition}\label{Definition-scalarReidemeisterTorsion} The  Reidemeister
	torsion of a chain complex $\CC$ with fixed choice of homology basis
	${\bf h}=({\bf h}_q)_q$ and ${\bf c}= ({\bf c}_q)_q$ is the positive real number
	\begin{equation}\label{Eq-scalarreidform}
		\tau(\CC;{\bf h};{\bf c}): = \prod_{q=0}^m [\pl_{q+1}({\bf b}_{q+1}) \;
		{\bf h}_q\;{\bf b}_{q}/{\bf c}_q]^{(-1)^q}\in \R_+.
	\end{equation}
\end{definition}

This is well-defined, since the number $\tau(\CC;{\bf h};{\bf c})$ does not depend on the choice of  ${\bf b}_q$ and ${\bf b}_{q+1}$.
It does depend on the choice of ${\bf h}$ and ${\bf c}$ though. We make the construction independent of ${\bf h}$
in the same manner as in Definition \ref{tors-norm}. We introduce the following notation
\begin{align*}
&\det\;{\bf h}_q\in \det\;H_q(\CC), \quad \det\;{\bf h} := \bigotimes_{q=0}^m \bl\det\; {\bf h}_q\br^{(-1)^q} \in \det\;H_\bullet(\CC).
\end{align*}
Choosing $\det\;{\bf h}$ as a unit element defines a norm $\| \cdot \|_{{\bf h}}$ on $\det\; H_\bullet(\CC)$.
We now define the abstract Reidemeister torsion of a complex in the same way as in Definition \ref{tors-norm}.

\begin{definition}
	The Reidemeister metric of $\CC$ is a norm on $\det\;H_\bullet(\CC)$ given by
	\begin{equation}\label{abstract-R}
		\|\cdot \|_{\det\;H_\bullet(\CC)}^R:=
		\tau(\CC;{\bf h};{\bf c}) \| \cdot \|_{\det \, {\bf h}}.
	\end{equation}
\end{definition}

The Reidemeister metric, viewed as a norm on $\det\;H_\bullet(\CC)$
is independent of the choice of ${\bf h}$. Now we apply this definition to $\overline{M}$.

\begin{definition}\label{Def-IRtorsion}
Consider $\overline{M}$ as a stratified  space with singular stratum $B$ and a tubular neighborhood $\overline{\cU}$
being a fibre bundle of spheres over $B$. Then \eqref{abstract-R} defines Reidemeister torsions on $\overline{M}$, that is
viewed either as a stratified space or as a smooth manifold, as follows
\begin{enumerate}
\item Consider a simplicial intersection chain complex $\CC:=(IC^{\mathfrak{p}}_\bullet(\overline{M}),
\partial_\bullet)$ on $\overline{M}$ as defined by \cite{GM1}, where $\mathfrak{p}$ stands for lower or upper middle
perversity.
Let ${\bf c}$ be a basis of chains as in Remark \ref{cell-basis} below. Then
\eqref{Eq-scalarreidform} defines  intersection
Reidemeister torsion $I^{\mathfrak{p}}\tau(\overline{M}, {\bf h})$, which is invariant under  subdivision by Dar \cite{Dar1}.
Moreover, \eqref{abstract-R} defines the intersection Reidemeister metric $\| \cdot \|^{IR}_{\overline{M},\mathfrak{p}}$.

\item Consider a simplicial chain complex $\CC:=(C_\bullet (\overline{M}),
\partial_\bullet)$.
Let ${\bf c}$ be a basis of simplexes. Then \eqref{Eq-scalarreidform} defines
Reidemeister torsion $\tau(\overline{M}, {\bf h})$, independent of the choice
of  subdivision, see e.g. Milnor \cite{Mil}.
Moreover, \eqref{abstract-R} defines the (classical) Reidemeister metric $\| \cdot \|^{R}_{\overline{M}}$.
\end{enumerate}
\end{definition}

\begin{remark}\label{cell-basis}
Since the  simplicial intersection chain complex $(IC^{\mathfrak{p}}_\bullet(\overline{M}),
\partial_\bullet)$ is a
subcomplex of the  chain complex $(C_\bullet (\overline{M}),
\partial_\bullet)$ we can not always
fix the preferred basis as in the smooth case. Therefore, we choose
as preferred basis each simplex that satisfies the intersection condition, if they
exist, and to complete this set of generators to a basis, we choose integral chains, in the sense of
the first named author and Spreafico \cite[Lemma 7.3]{HS2020}, \ie simplicial chains with coefficients integer numbers.
\end{remark}

	\begin{remark}
We may use a cellular chain complex in Definition \ref{Def-IRtorsion}. To do this, we need to extend the definition of the cellular intersection chain complex from \cite{HS2020} to a space with a wedge singularity. This extension should be feasible because of the rich structure of the tubular neighborhood, which reduces the problem to the cellular intersection chain complex of the fiber, which is a cone.
	\end{remark}

Since $\overline{M}$ is smooth, it is a homology manifold and the Poincaré duality holds. Thus the intersection homology
$IH^{\mathfrak{p}}_*(\overline{M}):= H_*(IC^{\mathfrak{p}}_\bullet(\overline{M}), \partial_\bullet)$ from Goresky and MacPherson \cite{GM1} coincides with the usual homology theory
$H_*(\overline{M}):=H_*(C_\bullet (\overline{M}), \partial_\bullet)$, and the de Rham integration map together with duality defines an isomorphism
\begin{align}\label{R}
R: \det \, H^*_{(2)}(M,g_\varepsilon) \equiv \det \, H^*_{(2)}(M,g)  \overset{\sim}{\to}  \det \, IH^{\mathfrak{p}}_*(\overline{M}) \equiv \det \, H_*(\overline{M}).
\end{align}
Nevertheless, the intersection
chain complex $(IC^{\mathfrak{p}}_\bullet(M), \partial_\bullet)$ needs not to be equal to the
chain complex $(C_\bullet (\overline{M}), \partial_\bullet)$.
Therefore a priori the corresponding Reidemeister torsions need not be the same and their equality is a theorem, which we now prove.

\begin{remark}
In the next theorem we present a direct proof of equality for both torsions. An alternative proof
uses a strong result from \cite{Dar1}. Namely, if $\overline{M}$ is smooth we can consider two different stratifications.
One can either consider $\Sigma = \emptyset$, in which case the intersection chain complex is the usual complex
$(C_\bullet (\overline{M}), \partial_\bullet)$. If we consider $\Sigma = B$ then the resulting chain complex is $(IC^{\mathfrak{p}}_\bullet(M), \partial_\bullet)$. By \cite[Theorem 2.9]{Dar1} the intersection Reidemeister torsion is independent of the stratification so the torsions of these two complexes coincide.
\end{remark}

\begin{theorem}\label{main-IR}
Consider a smooth manifold $\overline{M}$ with an embedded submanifold $B$ as above. Then
the intersection and the classical Reidemeister metric coincide
$$
\| \cdot \|^{IR}_{\overline{M}, \mathfrak{p}} = \| \cdot \|^{R}_{\overline{M}}.
$$
\end{theorem}

\begin{proof}
Both Reidemeister torsions satisfy a gluing formula, as in \cite[Theorem 3.2]{Mil}):
Consider $\overline{M} = M_1 \cup M_2$ a covering by two subsets and the gluing
sequence
\begin{equation}
\xymatrix{ 0 \ar@{^{(}->}[r] & M_1 \cap M_2 \ar[r] & M_1\sqcup M_2
\ar@{->>}[r] & \overline{M} \ar[r] &0}.
\end{equation}
The gluing formula for combinatorial torsions yields
\begin{equation}\label{scalar-gluing}
\begin{aligned}
&I^{\mathfrak{p}}\tau (\overline{M}; {\bf h}_M) = \frac{I^{\mathfrak{p}}\tau
(M_1; {\bf h}_{M_1})\cdot I^{\mathfrak{p}}\tau (M_2; {\bf h}_{M_2})}{\tau
(M_1 \cap M_2; {\bf h}_{M_1\cap M_2})\cdot \tau (\mathfrak{H})}, \\
&\tau (\overline{M}; {\bf h}_M) = \frac{\tau
(M_1; {\bf h}_{M_1})\cdot \tau (M_2; {\bf h}_{M_2})}{\tau
(M_1 \cap M_2; {\bf h}_{M_1\cap M_2})\cdot \tau (\mathfrak{H})}, \\
\end{aligned}
\end{equation}
where $\tau (\mathfrak{H})$ is the torsion of the
following long exact sequence
\begin{equation*}
\xymatrix@C=0.3cm{\cdot \cdot \ar[r] & H_q(M_1 \cap M_2) \ar[r] & H_q (M_1)
\oplus H_q(M_2) \ar[r] &  H_q(\overline{M})\ar[r] &
H_{q-1}(M_1\cap M_2)\ar[r] &\cdot \cdot}.
\end{equation*}
This can be reformulated as a gluing formula for Reidemeister torsion metrics as follows.
This long exact sequence defines an isomorphism between determinant lines in homology
\begin{equation}
\Phi: \det \, H_*(\overline{M}) \to \det \, H_*(M_1) \otimes \det \, H_*(M_2) \otimes \det \, H_*(M_1 \cap M_2)^{-1}.
\end{equation}
A tedious, but straightforward combinatorial computation shows
\begin{equation}\label{Phi}
\Phi (\det \, {\bf h}_M) = \tau (\mathfrak{H}) \cdot \det \, {\bf h}_{M_1}
\otimes \det \, {\bf h}_{M_1} \otimes \bigl( \det \, {\bf h}_{M_1\cap M_2} \bigr)^{-1}.
\end{equation}
From here we conclude from \eqref{scalar-gluing}, \eqref{abstract-R} and \eqref{Phi}
\begin{equation}\label{scalar-gluing-1}
\begin{aligned}
&\| \Phi^{-1} \bigl( \alpha \otimes \beta \otimes \gamma^{-1}\bigr) \|^{IR}_{\overline{M}, \mathfrak{p}}
= \frac{\| \alpha \|^{IR}_{M_1, \mathfrak{p}} \cdot \| \beta \|^{IR}_{M_2, \mathfrak{p}}}{\| \gamma \|^{IR}_{M_1\cap M_2, \mathfrak{p}} }, \\
&\| \Phi^{-1} \bigl( \alpha \otimes \beta \otimes \gamma^{-1}\bigr) \|^{R}_{\overline{M}}
= \frac{\| \alpha \|^{R}_{M_1} \cdot \| \beta \|^{R}_{M_2}}{\| \gamma \|^{R}_{M_1\cap M_2} },
\end{aligned}
\end{equation}
for any $\alpha \in  \det \, H_*(M_1)$, $\beta \in \det \, H_*(M_2)$ and $\gamma \in \det \, H_*(M_1 \cap M_2)$.
Thus, we can reduce our proof from the whole $\overline{M}$ to the tubular neighborhood $\overline{\cU}$ of the
singular stratum $B$. Cutting $\overline{\cU}$ into trivializable local bundles of cones
$C (\mathbb{S}^n)$, we can reduce the claim
to $W \times C (\mathbb{S}^n)$, where $W$ is any smooth manifold. The statement now
follows from Lemma \ref{product-lemma} below.
\end{proof}

\begin{lemma}\label{product-lemma}
Let $W$ be a smooth Riemannian manifold. Then
$$
\| \cdot \|^{IR}_{W \times C (\mathbb{S}^n),\mathfrak{p}} = \| \cdot \|^{R}_{W \times C
(\mathbb{S}^n)}.
$$
\end{lemma}

\begin{proof}
The K\"unneth equivalence of chain complexes, see King \cite[Theorem 4]{Kin} for the analogous statement on intersection complexes,
implies product formulas for combinatorial torsions
\begin{equation*}
\begin{aligned}
&\tau(W\times C (\mathbb{S}^n); {\bf h}_W\otimes {\bf h}_{C (\mathbb{S}^n)}) =  \bl \tau
(W;{\bf h}_W)\br^{\chi(C (\mathbb{S}^n))}\cdot \bl \tau(C (\mathbb{S}^n);{\bf h}_{C (\mathbb{S}^n)})\br^{\chi(W)}, \\
&I^{\mathfrak{p}}\tau(W\times C (\mathbb{S}^n); {\bf h}_W\otimes {\bf h}_{C (\mathbb{S}^n)}) =  \bl \tau
(W;{\bf h}_W)\br^{\chi(C (\mathbb{S}^n))}\cdot \bl I^{\mathfrak{p}}\tau(C
(\mathbb{S}^n);{\bf h}_{C (\mathbb{S}^n)})\br^{\chi(W)}.
\end{aligned}
\end{equation*}
Thus, the statement follows once we prove the following equality
\begin{align}\label{cone-torsion}
\tau(C (\mathbb{S}^n);{\bf h}_{C (\mathbb{S}^n)}) = I^{\mathfrak{p}}\tau(C (\mathbb{S}^n);{\bf h}_{C (\mathbb{S}^n)}).
\end{align}
The classical Reidemeister torsion on the left hand side can be computed similar to
de Melo, the first named author and Spreafico \cite[Proposition 2]{HMS}, so that
\[
\tau(C (\mathbb{S}^n);{\bf h}_{C (\mathbb{S}^n)})= [{\bf h}_0/{\bf n}_0],
\]
where ${\bf n}_0$ is a generator of $H_0(\mathbb{S}^n;\Z) = \Z$ over $\Z$. Now, the same result holds for
the intersection Reidemeister torsion with perversity $\mathfrak{p}$ by
\cite[Theorem 7.7]{HS2020}. Hence \eqref{cone-torsion} holds and the statement
follows.
\end{proof}

%%%%%%%%%%%%%%%%%%%%%%
\section{Cheeger-M\"uller theorem for wedges along a submanifold}
%%%%%%%%%%%%%%%%%%%%%%%

We can now prove our main result, Theorem \ref{main} on equality
of analytic and intersection Reidemeister torsions for manifolds with wedges along
an embedded submanifold. We phrase the result in terms of norms on determinant
lines in (co-)homology.

\begin{theorem}
	Let $(M,g_\epsilon)$ be a family of wedge space with singularity along an embedded
	submanifold $B$. If $\dim M$ and $\dim B$ are both odd, then the Cheeger-M\"uller theorem holds, \ie
	for $R$ induced by the de Rham integration map and duality as in \eqref{R}, and $\mathfrak{p}$ denoting
	lower or upper middle perversity, we have
	\begin{equation}
		\|\cdot\|^{RS}_{(M;g_{\epsilon})} = \| \, R(\cdot) \, \|^{IR}_{\overline{M},\mathfrak{p}}.
	\end{equation}
\end{theorem}

\begin{proof}
	The claim is a consequence of Theorems \ref{main-RS}
	and \ref{main-IR}, as well as the Cheeger-M\"uller theorem on closed smooth manifolds. Indeed,
	we have the following sequence of equations
\begin{equation*}
\|\cdot \|_{(M,g_\varepsilon)}^{RS} \stackrel{\textup{Theorem} \ \ref{main-RS}}{=} \|\cdot \|_{(M,g_0)}^{RS}
= \| \, R(\cdot) \, \|^{R}_{\overline{M}} \stackrel{\textup{Theorem} \ \ref{main-IR}}{=} \| \, R(\cdot) \, \|^{IR}_{\overline{M},\mathfrak{p}},
\end{equation*}
where the middle equality is simply the Cheeger-M\"uller theorem of Che\-eger \cite{Che1} and M\"uller \cite{Mul}.
\end{proof}

\section*{Declarations}

\subsection*{Funding}The first author was partially supported by FAPESP
2021/09534-4. The authors was partially supported by CAPES/DAAD:
8881.700909/ 2022-01.

\subsection*{Conflict of Interests} The authors declare no conflict of interest.

\bibliography{local}

\providecommand{\bysame}{\leavevmode\hbox to3em{\hrulefill}\thinspace}
\providecommand{\MR}{\relax\ifhmode\unskip\space\fi MR }
\providecommand{\MRhref}[2]{%
  \href{http://www.ams.org/mathscinet-getitem?mr=#1}{#2}
}
\providecommand{\href}[2]{#2}
\begin{thebibliography}{\textsc{ALMP12}}

\bibitem[\textsc{ALMP12}]{ALMP1}
\textsc{P.~Albin}, \textsc{E.~Leichtnam}, \textsc{R.~Mazzeo}, and
  \textsc{P.~Piazza}, \emph{The signature package on {W}itt spaces}, Ann. Sci.
  \'Ec. Norm. Sup\'er. (4) \textbf{45} (2012), no.~2, 241--310. \MR{2977620}

\bibitem[\textsc{ARS22}]{ARS}
\textsc{P.~Albin}, \textsc{F.~Rochon}, and \textsc{D.~Sher}, \emph{A
  {C}heeger-{M}\"{u}ller theorem for manifolds with wedge singularities}, Anal.
  PDE \textbf{15} (2022), no.~3, 567--642. \MR{4442836}

\bibitem[\textsc{AtLe13}]{AL}
\textsc{M.~Atiyah} and \textsc{C.~Lebrun}, \emph{Curvature, cones and
  characteristic numbers}, Math. Proc. Cambridge Philos. Soc. \textbf{155}
  (2013), no.~1, 13--37. \MR{3065256}

\bibitem[\textsc{BrLe93}]{BL-1}
\textsc{J.~Br\"{u}ning} and \textsc{M.~Lesch}, \emph{K\"{a}hler-{H}odge theory
  for conformal complex cones}, Geom. Funct. Anal. \textbf{3} (1993), no.~5,
  439--473. \MR{1233862}

\bibitem[\textsc{Che79}]{Che1}
\textsc{J.~Cheeger}, \emph{Analytic torsion and the heat equation}, Ann. of
  Math. (2) \textbf{109} (1979), no.~2, 259--322. \MR{528965}

\bibitem[\textsc{Che83}]{Che2}
\bysame, \emph{Spectral geometry of singular {R}iemannian spaces}, J.
  Differential Geom. \textbf{18} (1983), no.~4, 575--657 (1984). \MR{730920}

\bibitem[\textsc{Dar87}]{Dar1}
\textsc{A.~Dar}, \emph{Intersection {$R$}-torsion and analytic torsion for
  pseudomanifolds}, Math. Z. \textbf{194} (1987), no.~2, 193--216. \MR{876230}

\bibitem[\textsc{dMHS12}]{HMS}
\textsc{T.~de~Melo}, \textsc{L.~Hartmann}, and \textsc{M.~Spreafico}, \emph{The
  analytic torsion of a disc}, Ann. Global Anal. Geom. \textbf{42} (2012),
  no.~1, 29--59. \MR{2912667}

\bibitem[\textsc{Fra35}]{Fra}
\textsc{W.~Franz}, \emph{\"uber die torsion einer \"uberdeckung}, J. Reine
  Angew. Math. \textbf{173} (1935), 245--254.

\bibitem[\textsc{GoMa80}]{GM1}
\textsc{M.~Goresky} and \textsc{R.~MacPherson}, \emph{Intersection homology
  theory}, Topology \textbf{19} (1980), no.~2, 135--162. \MR{572580}

\bibitem[\textsc{HaSp20}]{HS2020}
\textsc{L.~Hartmann} and \textsc{M.~Spreafico}, \emph{Intersection torsion and
  analytic torsion of spaces with conical singularities}, 2020.
  \texttt{arXiv:2001.07801}

\bibitem[\textsc{Kin85}]{Kin}
\textsc{H.~C. King}, \emph{Topological invariance of intersection homology
  without sheaves}, Topology Appl. \textbf{20} (1985), no.~2, 149--160.
  \MR{800845}

\bibitem[\textsc{Lee18}]{Lee}
\textsc{J.~M. Lee}, \emph{Introduction to {R}iemannian manifolds}, Graduate
  Texts in Mathematics, vol. 176, Springer, Cham, 2018, Second edition of [
  MR1468735]. \MR{3887684}

\bibitem[\textsc{MaVe12}]{MaVe1}
\textsc{R.~Mazzeo} and \textsc{B.~Vertman}, \emph{Analytic torsion on manifolds
  with edges}, Adv. Math. \textbf{231} (2012), no.~2, 1000--1040. \MR{2955200}

\bibitem[\textsc{Mil66}]{Mil}
\textsc{J.~Milnor}, \emph{Whitehead torsion}, Bull. Amer. Math. Soc.
  \textbf{72} (1966), 358--426. \MR{0196736}

\bibitem[\textsc{M{\"{u}}l78}]{Mul}
\textsc{W.~M{\"{u}}ller}, \emph{Analytic torsion and {$R$}-torsion of
  {R}iemannian manifolds}, Adv. in Math. \textbf{28} (1978), no.~3, 233--305.
  \MR{498252}

\bibitem[\textsc{M{\"{u}}l93}]{Mul1}
\bysame, \emph{Analytic torsion and {$R$}-torsion for unimodular
  representations}, J. Amer. Math. Soc. \textbf{6} (1993), no.~3, 721--753.
  \MR{1189689}

\bibitem[\textsc{RaSi71}]{RS}
\textsc{D.~B. Ray} and \textsc{I.~M. Singer}, \emph{{$R$}-torsion and the
  {L}aplacian on {R}iemannian manifolds}, Advances in Math. \textbf{7} (1971),
  145--210. \MR{0295381}

\bibitem[\textsc{Rei35}]{Rei}
\textsc{K.~Reidemeister}, \emph{Homotopieringe und {L}insenr\"aume}, Abh. Math.
  Sem. Univ. Hamburg \textbf{11} (1935), no.~1, 102--109. \MR{3069647}

\end{thebibliography}
\bibliographystyle{amsalpha-lmp}

\end{document}